\documentclass[12pt]{amsart}

  \baselineskip 1 mm

\usepackage[osf,sc]{mathpazo}
\usepackage{amsmath,amsthm,amsfonts,latexsym,amsopn,verbatim,amscd,amssymb, tikz}
\usepackage{hyperref}
\usepackage{mathrsfs}

\theoremstyle{plain}
\newtheorem{thm}{Theorem}[section]
\newtheorem{lem}[thm]{Lemma}
\newtheorem{prop}[thm]{Proposition}

\theoremstyle{definition}

\numberwithin{equation}{section}

\newcommand{\bnum}{\begin{enumerate}}
\newcommand{\enum}{\end{enumerate}}

\begin{document}

\title{Representations induced from a  
normal subgroup of prime index}
\author[Soham Swadhin Pradhan]{Soham Swadhin Pradhan}
\address{Stat Math Unit, Indian Statistical Institute Bangalore}
\email{soham.spradhan@gmail.com}

\begin{abstract}
Let $G$ be a finite group, $H$ be a normal subgroup of prime index $p$. Let $F$ be a field of either characteristic $0$ or prime to $|G|$. Let $\eta$ be an irreducible $F$-representation of $H$. If $F$ is an algebraically closed field of characteristic either $0$ or prime to $|G|$, then the induced representation $\eta \uparrow^{G}_{H}$ is either irreducible or a direct sum of $p$ pairwise inequivalent irreducible representations. In this paper, we show that if $F$ is not  assumed  algebraically  closed field, then there are five possibilities in the decomposition of induced representation into irreducible representations. 
\end{abstract}
\maketitle
\section{Introduction} 
Let $G$ be a finite group, and $H$ be a normal subgroup of prime index $p$ in $G$. Let $F$ be a field of characteristic $0$ or prime to $|G|$ (not necessarily algebraically closed field). 
This paper is motivated by the problem: ``Given an irreducible $F$-representation $\eta$ of $H$, determine the irreducible sub-representations of the induced representation $\eta \uparrow^{G}_{H}$ induced from $\eta$".

Let $F$ be an algebraically closed field of characteristic either $0$ or prime to $|G|$.
We denote the set of  conjugacy classes in $G$ by 
$\mathcal C_G$ and the set of irreducible representations of $G$ by 
$\Omega_{G}$. For an element $g$ in $G$,
we denote by $C_{G}(g)$ the $G$-conjugacy class of $g$. 
Let $H$ be normal subgroup of $G$. If $h$ in $H$, let $C_{H}(h)$ 
denotes the $H$-conjugacy class of $h$. Then $G$ starts acting on 
$\mathcal C_H$ by conjugation.
For $h \in H$, and $g$ in $G$, the $g$-action on $\mathcal C_H$ 
maps $C_{H}(h)$ onto $C_{H}(ghg^{-1})$.
In fact, $C_G(h)$ is a union of the $H$-conjugacy classes
$C_{H}(ghg^{-1})$ for all $g$ in $G$. In the $G$-action on $\mathcal C_H$, 
$H$ itself acts 
trivially so we have actually a $G/H$-action on $\mathcal C_H$; 
the $G$-conjugacy leaves the complement $G-H$ of $H$ also invariant, 
and so $G-H$ is  also a union of certain $G$-conjugacy classes.
It is well known that, for a finite group $G$, $\Omega_G$ and 
$\mathcal C_G$ contain the same number of elements,
and these two sets are dual to each other. When $H$ is a 
normal subgroup of $G$, then $G$ also starts acting on $\Omega_H$ 
by conjugation. Let $\eta$ be in $\Omega_H$, and $g$ in $G$. Then 
$g\cdot \eta(h)=\eta(ghg^{-1})$ is also a representation of $H$, 
which is conjugate representation conjugating by $g$ in $G$.
This action of $G$ restricted to $H$ is trivial, so we actually 
have the action of $G/H$. The following theorem (see \cite{cl50}, Theorem $13.52$) gives decomposition of $\eta \uparrow^{G}_{H}$ into irreducible sub-representations. We call this theorem `index-$p$ theorem'.

\begin{thm}(Index-$p$ theorem)\label{index}
	Let $G$ be a group, $H$ a normal subgroup of index $p$, 
	$p$ a prime. Let $F$ be an algebraically closed field of charcteristic $0$ or prime to $|G|$. Let $\eta$ be an irreducible representation of $H$ over $F$.
	\begin{enumerate}
		\item If the $G/H$-orbit of $\eta$ is a singleton, then $\eta$ extends
		to $p$ mutually inequivalent representations 
		$\rho_1, \rho_2, \ldots, \rho_p$  of $G$.
		\item If the $G/H$-orbit of $\eta$ consists of $p$ points 
		$\eta = \eta_1, \eta_2,\ldots, \eta_p$
		then the induced representations $\eta_1\uparrow_H^G$, 
		$\eta_2\uparrow_H^G,\ldots, \eta_p\uparrow_H^G$, are equivalent,
		say $\rho$ and $\rho$ is irreducible.
	\end{enumerate}
\end{thm}

The following theorem is due to Berman (see \cite{berman120}), when $F$ is an algebraically closed, which gives us an expression of the primitive central idempotents in $F[G]$ in terms of the primitive central idempotents in $F[H]$.
\begin{thm}\label{berman}
Let $G$ be a finite group and $H$ be a normal subgroup of index $p$, a prime. Let $G/H = \langle {\bar{x}} \rangle$, for some $x$ in $G$. Let $F$ be an algebraically closed field of characteristic $0$ or prime to $|G|$. Let $\bar{C}(x)$ be the conjugacy class sum of $x$ in $F[G]$. Let $(\eta, W)$ be an irreducible $F$-representation of $H$ and $e_{\eta}$ be its corresponding primitive central idempotent in ${F}[H]$. 
We distinguish two cases:
\begin{itemize}
\item[(1)] If $e_{\eta}$ is a central idempotent in ${F}[G]$, 
then $\eta$ extends to $p$ distinct irreducible representations
$\rho_{0}, \rho_{1}, \dots , \rho_{p-1}$ (say) of $G$. 
It follows that $(\bar{C}(x))^{p}{e_{\eta}} = {\lambda}{e_{\eta}}$, where $\lambda \neq 0$.
For each $i$, $0 \leq i \leq p-1$, let $e_{\rho_{i}}$ be the primitive central idempotent corresponding to the representation $\rho_{i}$. Then
$${e_{\rho_{i}} = \frac{1}{p} 
\Big{(} 1 + \zeta^ic + \zeta^{2i} c^2 +\cdots +\zeta^{i(p-1)}c^{p-1}\Big{)} e_{\eta},}$$
where $c = \frac{\bar{C}{(x)}{e_{\eta}} }{\sqrt[p]{\lambda}}$ and $\zeta$ is a primitive $p^{th}$ root of unity in $F$.  
Moreover, $$e_{\eta} = e_{\rho_0} + e_{\rho_1} + \cdots + e_{\rho_{p-1}}.$$
		
\item[(2)] If $e_{\eta}$ is not a central idempotent in ${F}[G]$, then $\eta\uparrow_H^G, \eta^{x}\uparrow_H^G, \ldots ,\eta^{x^{p-1}}\uparrow_H^G$ are all equivalent to an irreducible representation $\rho$ (say) of $G$ over $F$ and in this case,
$$e_{\rho} = e_{\eta} + e_{\eta^{x}} + \cdots + e_{\eta^{x^{p-1}}}.$$
\end{itemize}
\end{thm}
{\bf{Notations.}} We need some notations to state our theorem. 
Let $G$ be a group, $H$ a normal subgroup of index $p$, 
$p$ a prime. Let $F$ be a field of characteristic $0$ or prime to $|G|$. Let $\bar{F}$ be the algebraic closure of $F$. 
Let $\Phi_{p}(X)$ denotes the $p$-th cyclotomic polynomial over $F$. Let $\zeta_{p}$ be a primitive $p$-th root of unity in $\bar{F}$.
Let $X^{p} - 1 = f_{0}(X)f_{1}(X) \dots f_{k}(X)$ with $f_{0}(X) = X - 1$ be the factorization into irreducibles over $F$. Let $\eta$ be an irreducible $F$-representation of $H$. Then by Schur's theory (see \cite{sc70}) on group representations
$$\eta \otimes_{F}{\bar{F}} \cong m(\eta_{1} \oplus \eta_{2} \oplus \dots \oplus \eta_{k}),$$
where $\eta_{i}'$s are Galois conjugates over $F$ and $m$ is Schur indices of $\eta_{i}$'s over $F$. One can show that either $\eta_{i} \ncong \eta^{x}_{i}$, for all $i$, $1 \leq i \leq k$ or $\eta_{i} \cong \eta^{x}_{i}$, for all $i$, $1 \leq i \leq k$. If $\eta_{i} \cong \eta^{x}_{i}$, for all $i$, $1 \leq i \leq k$, then by Theorem \ref{index}, each $\eta_{i}$ extends to $p$ distinct ways. For each $i$, let $\rho_{i,o}, \rho_{i, 1}, \dots , \rho_{i, p-1}$ be the $p$ extensions of $\eta_{i}$. Let $\chi_{\eta_{i}}$ be the character corresponding to the representation $\eta_{i}$. Let $e_{\eta_{i}}$ be the primitive central idempotent corresponding to $\eta_{i}$ in $\bar{F}[H]$. Let $e_{\rho_{i, j}}$ be the primitive central idempotent corresponding to $\rho_{i, j}$ in $\bar{F}[G]$. By Theorem \ref{berman}, $(\bar{C}(x))^{p}{e_{\eta_{i}}} = {\lambda_{i}}{e_{\eta_{i}}}$, where $\lambda_{i}$ is a nonzero element in $\bar{F}$. Let $\mu_{i}$ be a $p$-th root of $\lambda_{i}$ in $\bar{F}$. By Theorem \ref{berman}, for each $i$, the primitive central idempotent $e_{\rho_{i, j}}$ is:
$$e_{\rho_{i, j}} = \frac{1}{p}\Big\{1 + {\zeta_{p}}^{j}\big(\frac{\bar{C}(x)}{\mu_{i}}\big) + \dots + {\zeta_{p}}^{j(p-1)}
\big({\frac{\bar{C}(x)}{\mu_{i}}}\big)^{p-1}\Big\}
e_{\eta_{i}},$$
where $0 \leq j \leq p-1$.

Let $\mathscr{G}$ denotes Gal$(F(\chi_{\eta_{i}})/F)$. For $1 \leq i, j \leq k$, as $\eta_{i}$ and $\eta_{j}$ are Galois conjugate over $F$, there is an element $\sigma \in \mathscr{G}$ s.t. $\sigma(e_{\eta_{i}}) = e_{\eta_{j}}$. So $$\{\bar{C}(x)\}^{p}{e_{\eta_{j}}} = \{\bar{C}(x)\}^{p}\sigma(e_{\eta_{i}}) 
= \sigma\{\bar{C}(x)^{p}{e_{\eta_{i}}}\} = \sigma({\lambda_{i}}e_{\eta_{i}}) = \sigma(\lambda_{i}){e_{\eta_{j}}}.$$ 
Therefore, if $\sigma(e_{\eta_{i}}) = e_{\eta_{j}}$, then $\sigma(\lambda_{i}) = {\lambda_{j}}$, and so there is an element $\gamma$ in Gal$(\bar{F}/F)$ such that $\gamma(\mu_{i}) = \mu_{j}$. So one can conclude that either $\mu_{i} \in F$, for all $1 \leq i \leq k$, in this case $\mu _{1} = \mu_{2} = \dots = \mu_{k}$ or $\mu_{i} \notin F$, for all $1 \leq i \leq k$.
We use all the above notations to state our theorem.
\begin{thm}
Let $G$ be a finite group. Let $H$ be a normal subgroup of index $p$, a prime in $G$. Let $G/H = \langle{\bar{x}}\rangle$. Let $x$ be a lift of $\bar{x}$ in $G$. Let $F$ be a field of characteristic $0$ or prime to $|G|$. Let $\bar{F}$ be the algebraic closure of $F$. Let $\eta$ be an irreducible $F$-representation of $H$. Let 
$$\eta \otimes_{F}{\bar{F}} \cong m(\eta_{1} \oplus \eta_{2} \oplus \dots \oplus \eta_{k}),$$
where $\eta_{i}'$s are Galois conjugates over $F$ and $m$ is the Schur indices of $\eta_{i}$'s over $F$.
\begin{enumerate}
\item[1.] If $\eta \ncong \eta^{x}$, then $p$ number of inequivalent $H$-representatioons $\eta, \eta^{x}, \dots, \eta^{x^{p-1}}$ induce the same irreducible representation.

\item[2.]  If $\eta \cong \eta^{x}$ and $\eta_{i} \ncong \eta^{x}_{i}$, for all $i$, $1 \leq i \leq k$, then $\eta \uparrow^{G}_{H}$ is direct sum of a single irreducible $F$-representation of $G$.

\item[3.] If $\eta \cong \eta^{x}$, $\eta_{i} \cong \eta^{x}_{i}$ and $\mu_{i} \notin F$,  for all $i$, $1 \leq i \leq k$, then $\eta \uparrow^{G}_{H}$ is direct sum of a single irreducible $F$-representation of $G$.

\item[4.] If $\eta \cong \eta^{x}$, $\eta_{i} \cong \eta^{x}_{i}$, $\mu_{i} \in F$,  for all $i$, $1 \leq i \leq k$ and $\zeta_{p} \in F$, then $\eta \uparrow^{G}_{H}$ is direct sum of $p$ inequivalent irreducible $F$-representation of $G$.

\item[5.] If $\eta \cong \eta^{x}$, $\eta_{i} \cong \eta^{x}_{i}$, $\mu_{i} \in F$,  for all $i$, $1 \leq i \leq k$ and $\zeta_{p} \notin F$, then $\eta \uparrow^{G}_{H}$ is direct sum of $1 + k$ inequivalent irreducible $F$-representation of $G$.
\end{enumerate}
\end{thm}
\section{Preliminary Results}
We now state the following lemma without proof, which is immediate from the definitions.
\begin{lem}\label{lem}
Let $G$ be a finite group and $H$ be a normal subgroup of $G$. Let $F$ be a field of characteristic either $0$ or prime to $|G|$. Let $E$ be a Galois extension of $F$. Let $\gamma \in \textrm{Gal}(E/F)$. Let $\rho$ be an $E$-representation of $H$ and $\eta$ be an $F$-representation of $H$. Then 
\begin{center}
$(\eta \otimes_{F}{E})\uparrow^{G}_{H} \cong (\eta \uparrow^{G}_{H}) \otimes_{F}{E}$
\textrm{and} 
$(\rho^{\gamma} )\uparrow^{G}_{H} \cong (\rho \uparrow^{G}_{H})^{\gamma}.$
\end{center}
\end{lem}
\begin{prop}\label{decom}
Let $F$ a field of characteristic $0$ or prime to $n$.
Let $\displaystyle \Phi_n(X) = f_1(X) \dots f_k(X)$ be the decomposition of $\Phi_n(X)$ into irreducible factors over $F$. Then
\begin{enumerate}
\item Degrees of all $f_i(X)$'s are same.
\item Let $\zeta$ be a root of one $f_i(X)$. Then all the 
roots of $f_i(X)$ are $\{\zeta^{r_1}, \zeta^{r_2}, \ldots,\zeta^{r_s}\}$, 
where all $r_i$'s are natural numbers with $r_1=1$, and 
the sequence $\{ r_1, r_2,\ldots, r_s \}$ is independent of 
irreducible factors of $\Phi_n(X)$ and any root of $\Phi_n(X)$.
\end{enumerate}
\end{prop}
\begin{proof}\rm
	The splitting field of $\Phi_{n}(X)$ over $F$ is $F(\zeta)$.
	The Galois group Gal$(F(\zeta)/F)$, 
	permutes all the roots of $\Phi_n(X)$. The set all the roots of $f_i(X)$
	is: $\{\sigma(\zeta)\, | \, \sigma \in {\rm Gal}(F(\zeta)/F)\}$.
	So the set of all the roots of 
	$f_i(X)$ is: $\{\zeta^{r_1},\zeta^{r_2},\ldots,\zeta^{r_s}\}$, 
	where $r_i$'s are natural numbers with $r_1 = 1$. 
As $F$ is a field of characteristic $0$ or prime to $n$, therefore the degree 
of  $f_i(X)$ is $s$, and it is clear that the degrees of  $f_{i}(X)$'s are same. 
It is also clear that the sequence $\{{r_1} = 1, r_2,\cdots,r_s\}$ is independent of roots of  $f_i(X)$ up to the order and also independent of irreducible factors of $\Phi_n(X)$.
\end{proof}
\section{Proof of Theorem $1.3$}
\subsection{Proof of (1)}
If  
$(\eta, W) \ncong (\eta^x, W^x)$, then all the irreducible representations 
$\{(\eta^{x^i},W^{x^i})\}_{i=0}^{p-1}$ are pairwise inequivalent. 
So $(\eta, W)$ is not 
equivalent to $(\eta^s, W^s)$ for any
$s\in {G - H}$. Hence, by the Mackey irreducibility criterion 
(see \cite{cur93}, Theorem $45.2$)
$\eta \uparrow_{H}^{G}$ is irreducible and also
\begin{center}
$ \eta \uparrow _{H}^{G} \cong \eta^x \uparrow _{H}^{G} \cong 
\cdots \cong \eta^{x^{p-1}}\uparrow _{H}^{G}.$
\end{center}
This completes the proof of the statement $(1)$.

Suppose that $\eta \cong \eta^{x}$. Let 
$$\eta \otimes_{F}{\bar{F}} \cong m(\eta_{1} \oplus \eta_{2} \oplus \dots \oplus \eta_{k}),$$
where $\eta_{i}'$s are Galois conjugates over $F$ and $m$ is Schur indices of $\eta_{i}$'s over $F$. By lemma \ref{lem} one can show that following two cases can arise:
\begin{enumerate}
	\item[(i)] $\eta_{i} \ncong \eta^{x}_{i}$, for all $i$, $1 \leq i \leq k$;
	\item[(ii)] $\eta_{i} \cong \eta^{x}_{i}$, for all $i$, $1 \leq i \leq k$.
\end{enumerate}
\subsection{Proof of (2)}
If $\eta_{i} \ncong \eta^{x}_{i}$, for all $i$, $1 \leq i \leq k$. Write 
\begin{align*}
\eta \otimes_{F}{\bar{F}} & \cong m (\eta_{1} \oplus  \eta^{x}_{1}\oplus \dots \oplus \eta^{x^{p-1}}_{1}) \\
&\oplus m (\eta_{2} \oplus \eta^{x}_{2} \oplus \dots \oplus \eta^{x^{p-1}}_{2}) \\
& \oplus \hdots  \hspace{.5 cm} \hdots \hspace{.7 cm}  \hdots \hspace{.7cm} \hdots\oplus\\
&\oplus m (\eta_{l} \oplus \eta^{x}_{l} \oplus \dots \oplus \eta^{x^{p-1}}_{l}).
\end{align*}
Note that $lp = k$. By Lemma \ref{lem}
$$(\eta \uparrow^{G}_{H}) \otimes_{F}{\bar{F}} \cong (\eta \otimes_{F}{\bar{F}})\uparrow^{G}_{H} \cong mp(\eta_{1}\uparrow^{G}_{H} \oplus \eta_{2}\uparrow^{G}_{H} \oplus \dots \oplus \eta_{l}\uparrow^{G}_{H}).$$
Notice that $\eta_{1}\uparrow^{G}_{H}, \eta_{2}\uparrow^{G}_{H}, \dots, \eta_{l}\uparrow^{G}_{H}$ are Galois conjugate over $F$. So, $\eta \uparrow^{G}_{H}$ is equivalent to direct sum of a single irreducible $F$-representation. This completes the proof of the statement $(2)$.

If $\eta_{i} \cong \eta^{x}_{i}$, for all $i$, $1 \leq i \leq k$, then each $\eta_{i}$ extends to $p$ distinct ways. For each $i$, let $\rho_{i,o}, \rho_{i, 1}, \dots , \rho_{i, p-1}$ be the $p$ extensions of $\eta_{i}$. Let $\chi_{\eta_{i}}$ be the character corresponding to the representation $\eta_{i}$. Let $e_{\eta_{i}}$ be the primitive central idempotent corresponding to $\eta_{i}$ in $\bar{F}[G]$. Let $e_{\rho_{i, j}}$ be the primitive central idempotent corresponding to $\rho_{i, j}$ in $\bar{F}[G]$. By Berman's theorem (see Theorem \ref{berman}), $(\bar{C}(x))^{p}{e_{\eta_{i}}} = {\lambda_{i}}{e_{\eta_{i}}}$, where $\lambda_{i}$ is a nonzero element in $\bar{F}$. Let $\mu_{i}$ be a $p$-th root of $\lambda_{i}$ in $\bar{F}$. By Berman's theorem (see Theorem \ref{berman}), for a fixed $i$, the primitive central idempotent corresponding to $\rho_{i, j}$ is:
$$e_{\rho_{i, j}} = \frac{1}{p}\big{\{}1 + {\zeta_{p}}^{j}\big{(}\frac{\bar{C}(x)}{\mu_{i}}\big{)} + \dots + {\zeta_{p}}^{j(p-1)}\big{(}{\frac{\bar{C}(x)}{\mu_{i}}}^{p-1}\big{)}\big{\}}
e_{\eta_{i}},$$
where $0 \leq j \leq p-1$.
\subsection{Proof of (3)} If $\mu_{i} \notin F$, for $1 \leq i \leq k$. In this case, $\rho_{i, j}$'s are Galois conjugate over $F$. So the induced representation $\eta \uparrow^{G} _{H}$ is direct sum of single irreducible $F$-representation. This completes the proof of the statement $(3)$.
\subsection{Proof of (4)}
If  $\mu_{i} \in F$, for $1 \leq i \leq k$. In this case, $\mu _{1} = \mu_{2} = \dots = \mu_{k} = \mu$ (say). 
If $\mu \in F, \zeta_{p} \in F$, then the induced representation $\eta \uparrow^{G}_{H}$ is direct sum of $p$ irreducible $F$-representations. Let $\psi_{j}$, $j = 0, 2, \dots , p-1$ be the $p$ irreducible $F$-representations. Then
$$\psi_{j} = m_{j} ({\rho_{1,j}} \oplus {\rho_{2, j}} \oplus \dots \oplus {\rho_{k, j}}),$$
where $0 \leq j \leq p-1$, and $m_{j}$ is the Schur index of each $\rho_{i, j}$. So, $$\eta \uparrow^{G}_{H} = \bigoplus^{p - 1}_{j = 0}(m/m_{j})\psi_{j}.$$
This completes the proof of the statement $(4)$.

\subsection{Proof of (5)} 
If  $\mu_{i} \in F$, for $1 \leq i \leq k$. In this case, $\mu _{1} = \mu_{2} = \dots = \mu_{k} = \mu$ (say). 
Assume that $\mu \in F, \zeta_{p} \notin F$. As $\mu \in F$, then ${\rho_{1,o}}, {\rho_{2, o}}, \dots , {\rho_{k, o}}$ are Galois conjugate over $F$, and the Schur index is equal to $m$. So, $\eta$ extends to an irreducible $F$-representation $m ({\rho_{1,o}} \oplus {\rho_{2, o}} \oplus \dots \oplus {\rho_{k, o}})$, say $\rho$. Then 
$$\eta \uparrow^{G}_{H} = F[G/H] \otimes_{F} \rho.$$ Let $\Phi_{p}(X) = f_{1}(X)f_{2}(X) \dots f_{k}(X)$ be the factorization into irreducible polynomials over $F$. Then $X^{p} - 1 = (X - 1)f_{1}(X)f_{2}(X) \dots f_{k}(X)$ be the factorization into irreducibles over $F$. So
$F[G/H] \approx F[X]/\langle{X - 1}\rangle \oplus^{k}_{i = 1} F[X]/\langle{f_{i}(X)}\rangle.$ In this case, $\eta \uparrow^{G}_{H}$ is direct sum of $1 + k$ many distinct irreducible $F$-representations of $G$.
This completes the proof of the statement $(5)$.

\end{document}